\let\accentvec\vec 
\documentclass{article}

\let\vec\accentvec 

\usepackage{amssymb,amsmath,amsthm}
\usepackage{mathrsfs}
\usepackage{url}  

\usepackage{hyperref} 
\usepackage{amsrefs}

\newcommand{\la}{\langle}
\newcommand{\ra}{\rangle}

\newcommand{\Iff}{\Leftrightarrow}
\renewcommand{\iff}{\leftrightarrow}
\newcommand{\Implies}{\Rightarrow}

\newcommand{\mc}{\mathcal}

\newcommand{\mf}{\mathfrak}

\newcommand{\inter}{\cap}

\renewcommand{\to}{\rightarrow}

\newcommand{\restrict}{\upharpoonright}

\renewcommand{\and}{\,\&\,}

\newcommand{\Q}{\mathbb{Q}}

\newcommand{\nil}{\varnothing}

\newtheorem{thm}{Theorem}
\newtheorem{cor}[thm]{Corollary}
\newtheorem{lem}[thm]{Lemma}
\newtheorem{pro}[thm]{Proposition}

\theoremstyle{definition}
\newtheorem{df}[thm]{Definition}
\theoremstyle{remark}
\newtheorem{rem}[thm]{Remark}
\newtheorem{que}[thm]{Question}

\begin{document}
	\title{Numberings and randomness\footnote{The final publication is available at Springer via
	\href{http://dx.doi.org/10.1007/978-3-642-03073-4_6}{http://dx.doi.org/10.1007/978-3-642-03073-4\_6}}}
	\author{Katie Brodhead \\ Bj\o rn Kjos-Hanssen\footnote{Partially supported by NSF grant DMS-0652669.}}
	\maketitle

	\begin{abstract} 
	We prove various results on effective numberings and Friedberg numberings of families related to algorithmic randomness. The family of all Martin-L\"of random left-computably enumerable reals has a Friedberg numbering, as does the family of all $\Pi^0_1$ classes of positive measure. On the other hand, the $\Pi^0_1$ classes contained in the Martin-L\"of random reals do not even have an effective numbering, nor do the left-c.e. reals satisfying a fixed randomness constant. For $\Pi^0_1$ classes contained in the class of reals satisfying a fixed randomness constant, we prove that at least an effective numbering exists.
	\end{abstract}

	\section{Introduction}
	The general theory of numberings was initiated in the mid-1950s by Kolmogorov, and continued by Mal'tsev and Ershov
	\cite{Ers99}.  A \emph{numbering}, or \emph{enumeration}, of a collection $C$ of objects is a surjective map $F:\omega \rightarrow
	C$. In one of the earliest results, Friedberg \cite[1958]{Fri58} constructed an injective numbering $\psi$ of the $\Sigma^0_1$ or computably
	enumerable (\emph{c.e.})  sets such that the relation ``$n \in\psi(e)$'' is itself $\Sigma^{0}_{1}$. In a more general and informal sense,  a numbering $\psi$ of a collection of objects all having complexity $\mathcal{C}$ (such as $n$-c.e., $\Sigma^{0}_{n}, \mbox{ or
	}\Pi^{0}_{n}$) is called \emph{effective} if the relation ``$x\in\psi(e)$'' has complexity $\mathcal{C}$. If in addition the numbering is injective, then it is called a \emph{Friedberg numbering}.  

	Brodhead and Cenzer~\cite{BrC08} showed that there is an effective Friedberg numbering of the $\Pi^{0}_{1}$ classes in Cantor space $2^\omega$. They showed that effective numberings exist of the $\Pi^0_1$ classes that are {homogeneous}, and {decidable}, but not of the families consisting of $\Pi^0_1$ classes that are of measure zero, thin, perfect thin, small, very small, or nondecidable, respectively.  

	In this article we continue the study of existence of numberings and Friedberg numberings for subsets of $\omega$ and $2^\omega$. Many of our results are related to algorithmic randomness and in particular Martin-L\"of randomness; see the books of Li and Vit\'anyi \cite{LV} and Nies \cite{Nies:book}.  

	We now outline some notation and definitions used throughout. A subset $T$ of $2^{<\omega}$ is a \emph{tree} if it is closed under prefixes.  The set $[T]$ of infinite paths through $T$ is defined by $X \in [T] \iff (\forall n) X \restrict n \in T$, where $X\restrict n$ denotes the initial segment $\la X(0),X(1),\hdots,X(n-1)\ra$. Next, $P$ is a $\Pi^{0}_{1}$ class if $P = [T]$ for some computable tree $T$. Let $\sigma^{\frown}\tau$ denote the concatenation of $\sigma$ with $\tau$ and let $\sigma^{\frown}i$ denote $\sigma ^{\frown}\la i\ra$ for $i \in \omega$. The prefix ordering of strings is denoted by $\preceq$, so we have $\sigma\preceq\sigma^\frown\tau$.  The string $\sigma\in T$ is a \emph{dead end} if no extension $\sigma^{\frown}i$ is in $T$. For any $\sigma \in 2^{<\omega}$, $[\sigma]$ is the cone consisting of all infinite sequences extending $\sigma$. For a set of strings $W$, $[W]^\preceq=\bigcup_{\sigma\in W} [\sigma]$.

	\section{Families of left-c.e.\ reals}

	\subsection{Basics}

	For our definition of left-c.e. reals we will follow the book of Nies \cite{Nies:book}. Let $\mathbb Q_2$ be the set of dyadic rationals $\{\frac{a}{2^b} \le 1: a,b\in\omega\}$. For a dyadic rational $q$ and real $x\in 2^\omega$, we say that $q<x$ if $q$ is less than the real number $\sum_{i\in\omega} x(i) 2^{-(i+1)}$.
	\begin{df}
	A real $x\in 2^\omega$ is \emph{left-c.e.} if $\{q\in\Q_2: q<x\}$ is c.e.
	\end{df}
	 Let $\le_L$ denote lexicographic order on $2^\omega$. 
	A dyadic rational may be written in the form $q=\sum_{i=1}^n a_i 2^{-i}$ where $a_n=1$, and each $a_i\in\{0,1\}$. The \emph{associated binary string} of $q$ is $s(q)=\la a_1,\ldots, a_n\ra$. (If $q=0$ then $n=0$ and the associated string is the empty string.) Conversely, the associated dyadic rational of $\sigma\in 2^{<\omega}$ is $\sum_{i=0}^{|\sigma|-1} \sigma(i) 2^{-(i+1)}$.

	\begin{lem}
	For each $x\in 2^\omega$, we have that $$\{q\in\Q_2:q<x\}\text{ is c.e.} \Iff \{\sigma\in 2^{<\omega}: \sigma^\frown 0^\omega<_L x\}\text{ is c.e.}$$
	\end{lem}
	\begin{proof}
	We have that $\sigma^\frown 0^\omega<_L x$ iff the associated dyadic rational of $\sigma$ is less than $x$, and $q<x$ iff the associated binary string $\sigma$ of $q$ satisfies $\sigma^\frown 0^\omega<_L x$. In fact, $\{s(q): q\in\mathbb Q_2, q<x\}=\{\sigma: \sigma^\frown\omega<_L x\}$.
	\end{proof}

	\begin{df}
	An \emph{effective numbering} of a family of left-c.e.\ reals $\mathcal R$ is an onto map $r:\omega\mapsto\mc R$ such that 
	$$\{(q,e)\in\Q_2\times\omega \mid q<r(e)\}$$
	is c.e.\ If $r$ is also injective then $r$ is called a \emph{Friedberg numbering} of $\mathcal R$.
	\end{df}

	\begin{thm}\label{piolunu}
	The family of all left-c.e.\ reals has an effective numbering.
	\end{thm}
	\begin{proof}
	Let $W_{e,s}$ be the $e^\text{th}$ c.e.\ subset of $\Q_2$ as enumerated up to stage $s$. Let $r_{e,s}$ be the greatest element of $W_{e,s}$ and let $r(e)=\lim_{s\to\infty} r_{e,s}$. It is easy to check that $r$ is an effective numbering of $\mathcal R$. 
	\end{proof}

	Some notions from algorithmic randomness will be needed repeatedly below. $\Omega$ is any fixed Martin-L\"of random left-c.e.\ real with computable approximation $\Omega_s\le_L\Omega_{s+1}$, $s\in\omega$. Let $K$ denote prefix-free Kolmogorov complexity. Schnorr's Theorem states that a real $x\in 2^\omega$ is Martin-L\"of random if and only if there is a constant $c$ such that for all $n$, $K(x\restrict n)\ge n-c$.

	\begin{thm}\label{gmt}
	The family of all Martin-L\"of random left-c.e.\ reals has an effective enumeration.
	\end{thm}
	\begin{proof}
	Let $K_t$ a uniformly computable approximation to Kolmogorov complexity at stage $t$, satisfying $K_{t+1}\le K_t$. To obtain an enumeration of the Martin-L\"{o}f random left-c.e.\ reals, it suffices to enumerate all Martin-L\"{o}f random left-c.e.\ reals $y$ such that $K(y\restrict n)\geq n-c$ for all $n$, uniformly in $c$.  

	  Initially our $m^\text{th}$ ML-random left-c.e.\ real $m_e$ will look like $r_e=r(e)$ from Theorem \ref{piolunu}, i.e.\ $m_{e,s}=r_{e,s}$ unless otherwise stated. Let $r_{e,t}[n]$ be the associated string, restricted or appended with zeroes if necessary to obtain length $n$. If at some stage $t$, for some $n=n_t\in\omega$, 
	$$K_t(r_{e,t}[n]) <n-c,$$ 
	then let $m_{e,s}=r_{e,t}[n]^{\frown}\Omega_s$
	at all stages $s>t$ until, if ever, $K_s(r_{e,s}[n])\geq n-c$ at some stage $s>t$.  At this point, $r_{e,t}[n]< r_{e,s}[n]$, 
	 since $r_e$ is a left-c.e.\ real.  Resume where we left off in defining $m_e=r_e$, starting immediately at stage $s$ with $m_{e,s}=r_{e,s}$.  
	This process continues for the entire construction of each $m_e$.  

	This enumeration contains all left-c.e.\ reals which are Martin-L\"{o}f random with respect to the constant $c$, and only Martin-L\"{o}f left-c.e.\ random reals. Thus the merger of these enumerations over all $c$ is an enumeration of all Martin-L\"of random left-c.e. reals.
	\end{proof}

	\subsection{Kummer's method}

	Kummer \cite[1990]{Kummer} gave a priority-free proof of Friedberg's result.  The conditions set forth in the proof provide a method of obtaining Friedberg numberings.

	A \emph{c.e.\ class} is a uniformly c.e.\ collection of subsets of $\omega$ (or equivalently, of $2^{<\omega}$ or $\Q_2$).

	\begin{thm}[Kummer \cite{Kummer}]
	If a c.e.\ class can be partitioned into two disjoint c.e.\ subclasses $L_1$ and $L_2$ such that $L_1$ is injectively enumerable and contains infinitely many extensions of every finite subset of any member of $L_2$, then the class is injectively enumerable.
	\end{thm}

	\begin{thm}\label{left.ce.reals}
	There is a Friedberg numbering of the left-c.e.\ reals.
	\end{thm}
	\begin{proof}
	Let $C(x)=\{\tau: \tau^\frown 0^\omega <_L x \}$.
	Let 
	$$\mathcal L=\{C(x): x\text{ is left-c.e.}\},$$
	$$L_1=\{C(x): x(n)=1\text{ for an odd finite number of }n\},$$
	and $L_2=\mathcal L\setminus L_1$.
	It is clear that $L_1$ is injectively enumerable, and each finite subset $F$ of a member of $L_2$ is contained in infinitely many members of $L_1$. The non-trivial part is to see that $L_2$ is c.e. Briefly, the idea is that we modify an enumeration $\{r_e\}_{e\in\omega}$ of all left-c.e.\ reals to only allow 1s to be added and removed in pairs of two. That is, we let $r^*_{e,s}$ be the longest prefix $\sigma$ of the string associated with $r_{e,s}$ such that the number of 1s in $\sigma$ is even. If in the end there are infinitely many 1s in $r_e$ then $r^*_e=r_e$, and it is clear that $r^*_{e,s}\le r^*_{e,s+1}$.
	\end{proof}

	\begin{thm}\label{sodexho}
	There is a Friedberg numbering of the Martin-L\"of random left-c.e.\ reals.
	\end{thm}
	\begin{proof}
	Let 
	$$\mathcal R=\{C(x): x \text{ is ML-random and left-c.e.}\},$$
	$L_1=\{C(1^{n\frown}\Omega):n\in\omega\}$, and $L_2=\mathcal{R}\setminus L_1$.  
	Again, it is clear that $L_1$ is injectively enumerable and each finite subset of a member of $L_2$ can be extended to infinitely many members of $L_1$. We will argue that $L_2$ is c.e. Note that $1^{n\frown}\Omega<_L 1^{n+1\frown}\Omega$ for each $n$. Thus
	$$L_2=\bigcup_{n\in\omega}\left\{C(y)\in\mc R:  1^{n\frown}\Omega<_L y<_L 1^{n+1\frown}\Omega\right\}.$$ 
	so it suffices to show that the sets 
	\begin{equation}
	\tag{1} \{C(y)\in\mc R: y<_{L}1^{n\frown}\Omega\} \text{,}
	\end{equation}
	\begin{equation}
	\tag{2} 
	\{C(y)\in\mc R: 1^{n\frown}\Omega <_L y \}
	\end{equation}
	are uniformly c.e.

	Notice that $y<_L 1^{n\frown}\Omega$ iff there is some $k$ such that $y\restrict k <_L (1^{n\frown}\Omega)\restrict k$, so for (1) it suffices to show that $\{C(y)\in\mc R: y\restrict k <_L (1^{n\frown}\Omega)\restrict k\}$ is c.e., uniformly in $n$ and $k$. This is non-trivial only if $k>n$, and in fact it suffices to show that a suitable subfamily $\mc F_k$ of $\left\{C(y)\in\mc R: y<_{L}\Omega\right\}$ containing
	\begin{equation}
	\tag{$1'$} \{C(y)\in\mc R: y\restrict k<_L \Omega\restrict k\}
	\end{equation} 
	is uniformly c.e.\ for $k\in\omega$. 

	We modify the enumeration $\{m_e\}_{e\in\omega}$ of the left-c.e.\ random reals from Theorem \ref{gmt}, producing a new enumeration $\{\widehat m_e\}_{e\in\omega}$.  Initially, as long as $\Omega\restrict k$ looks like the constant-zero string $0^k$ then $\widehat m_e$ is made to look like $0^\frown \Omega$. Note that since $\Omega\ne 0^\omega$, $0^\frown \Omega<_L \Omega$. 

	If at any stage it looks like $\Omega\restrict k\ne 0^k$ then thereafter we let $\widehat m_e=m_e$ as long as $m_e\restrict k<_L\Omega\restrict k$. If at some stage $s$, $m_{e,s}\restrict k\ge_L \Omega_s\restrict k$,  then we say that we are in an undesirable state, and we let $\widehat m_{e,t}=m_{e,s-1}\restrict k^\frown \Omega_t$ for all $t\ge s$ until a possible later stage where we are in a desirable state again. 

	Thus, if $m_e$ really satisfies $m_e\restrict k<_L\Omega\restrict k$ then we will have $\widehat m_e=m_e$, and if not then $\widehat m_e$ will be a finite string $\sigma<_L\Omega\restrict k$ followed by $\Omega$, so in any case it will be a Martin-L\"of random real. Thus $\{C(\widehat m_e)\}_{e\in\omega}$ is an effective enumeration of a family $\mc F_k$ as stated. The argument for (2) is analogous. 
	\end{proof}

	\subsection{Specifying randomness constants}

	Recall that Schnorr's Theorem states that a real $x\in 2^\omega$ is Martin-L\"of random if and only if there is a constant $c$ such that for all $n$, $K(x\restrict n)\ge n-c$. The optimal randomness constant of $x$ is the least $c$ such that this holds. For each interval $I\subseteq\omega$ we let $\mc A_I$ ($\mc R_I$) denote the set of all Martin-L\"of random (and left-c.e., respectively) reals whose optimal randomness constant belongs to $I$.
	Let $\mu$ denote the fair-coin Cantor-Lebesgue measure on $2^\omega$. 
	By the proof of Schnorr's Theorem we have
	$$\mu(\{x:(\forall n)K(x\restrict n)\geq n-c\})\geq 1-2^{-(c+1)}.$$
	Consequently, if $c\geq 0$, then  $\mu \mc A_{[0,c]}>0$ and $\mc A_{[0,c]}\ne\nil$.

	\begin{thm}\label{ten}
	Let $c\ge 0$. There is no effective enumeration of $\mathcal R_{[0,c]}$.
	\end{thm}
	\begin{proof}
	Suppose that $\{\alpha_e\}_{e\in\omega}$ is such an enumeration, with a uniformly computable approximation $\alpha_{e,s}$ such that $\alpha_e=\lim_{s\to\infty}\alpha_{e,s}$ and $\alpha_{e,s}\le\alpha_{e,s+1}$. Note that
	\[\mc A_{[0,c]}=\{x:(\forall n)K(x\restrict n)\geq n-c\}\]
	is a $\Pi^0_1$ class. Let $\beta_{s}=\max\{\alpha_{e,s}: e\le s\}$. Then $\beta=\lim_{s\to\infty}\beta_s$ is left-c.e., and since the left-c.e.\ members of $A_{[0,c]}$ are dense in $A_{[0,c]}$, $\beta$ is the rightmost path of $A_{[0,c]}$. However the rightmost path of a $\Pi^0_1$ class is also \emph{right-c.e.}, defined in the obvious way. Thus $\beta$ is a Martin-L\"of random real that is computable, a contradiction.
	\end{proof}

	\begin{thm}\label{nine}
	For each $c$ there is an effective numbering of $\mc R_{[c+1,\infty)}$. 
	\end{thm}
	\emph{Proof sketch.} 
	Let $\{m_e\}_{e\in\omega}$ be an effective enumeration of all left-c.e.\ random reals, with the additional property that for each $e$ there are infinitely many $e'$ such that for all $s$, $m_{e,s}=m_{e',s}$. We will define an effective numbering $\{\alpha_e\}_{e\in\omega}$ of $\mc R_{[c+1,\infty)}$. 

	We say that a string $\sigma$ \emph{satisfies randomness constant $c$ at stage $t$} if
	$$K_t(\sigma)\ge |\sigma|-c;$$
	otherwise, we say that $\sigma$ fails randomness constant $c$ at stage $t$.

	We proceed in stages $t\in\omega$, monitoring each $m_{e,t}$ for $e\le t$ at stage $t$. If for some $t_0$, $n$, $e$, 
	we observe that $m_{e,t_0}[n]$ fails randomness constant $c$, then we want to assign a place for $m_e$ in our enumeration of $\mc R_{[c+1,\infty)}$. So we let $d$ be minimal so that $\alpha_d$ has not yet been mentioned in the construction, and let $\alpha_{d,s}= m_{e,s}$ for all stages $s\ge t_0$ until further notice. If $m_{e,t_1}[n]$ at some stage $t_1\ge t_0$ satifies randomness constant $c$, then we \emph{regret} having assigned $m_e$ a place in our enumeration $\{\alpha_e\}_{e\in\omega}$. To compensate for this regret, we choose a large number $p=p_{c,n}$ and for all stages $s\ge t_1$ let $\alpha_{d,s}=m_{e,s}[n]^\frown 0^{p\frown}\Omega_s$. The largeness of $p$ guarantees that $m_{e,s}[n]^\frown 0^p$ does not satisfy randomness constant $c$. \footnote{To be precise, if $|\sigma|=n$ then there are universal constants $\hat c$ and $\tilde c$ such that, thinking of $p$ sometimes as a string,  \label{fn}
	$K(\sigma^\frown 0^p)\le K(\sigma)+K(p)+\hat c \le 2|\sigma| + 2 |p| + \tilde c=2n + 2\log p +\tilde c \le n+p-c$
	provided $p-2\log p\ge n+\tilde c+c,$
	which is true for $p=p_{n,c}$ that we can find effectively. }
	If $m_e$ actually does fail randomness constant $c$, but at a larger length $n'>n$, then because there are infinitely many $e'$ with $m_{e'}=m_e$ we will eventually assign some $\alpha_{d'}$ to some such $m_{e'}$ at a stage $t_2$ that is so large that $m_{e',t_2}[n']=m_{e'}[n']$. Thus, each real in $\mc R_{[c+1,\infty)}$ will eventually be assigned a permanent $\alpha_{d'}$. 

	\begin{rem}
	We believe that one can even show that there is a Friedberg numbering of $\mc R_{[c+1,\infty)}$. The idea is to modify $L_1$ so that the strings $1^n$ are replaced by $1^{d_c+n}$ for a sufficiently large $d_c$, as in the footnote on page \pageref{fn}.
	\end{rem}

	\begin{rem}
	Theorems \ref{ten} and \ref{nine} indicate perhaps that the left-c.e.\ members of $\Sigma^0_2$ classes are generally easier to enumerate than those of $\Pi^0_1$ classes; this may be due to the ``$\Sigma^0_n$ nature'' of left-c.e.\ reals (for $n=1$).
	\end{rem}

	\begin{figure}
	\begin{center}
	\begin{tabular}{|c|c|c|}
	\hline
	Family      & Enumeration? & Friedberg? \\
	\hline
	All $\Pi^0_1$ classes             &      & Yes, by Theorem \ref{all} \\
	\hline
	All left-c.e. reals        &    & Yes, by Theorem \ref{left.ce.reals}\\
	\hline
	$\Pi^0_1$ classes $C$, $\mu C>0$ & & Yes, by Theorem \ref{all} \\

	\hline
	Left-c.e. reals in \textsc{MLR}   &    & Yes, by Theorem \ref{sodexho} \\
	\hline
	$\Pi^0_1$ classes $\subseteq \mc A_{[0,c]}$ & Yes, by Proposition \ref{lastmin}  & Open problem \\

	\hline
	$\Pi^0_1$ classes $\subseteq\textsc{MLR}$  & No, by Theorem \ref{1}     & \\

	\hline
	Left-c.e.\ reals in $ \mc A_{[0,c]}$  & No, by Theorem \ref{ten} &  \\
	\hline
	\end{tabular}
	\end{center}
	\caption{Existence of effective numberings and Friedberg numberings, where $\textsc{MLR}=\bigcup_{c\in\omega} \mc A_{[0,c]}$.}
	\end{figure}

	Whether a set of the form $\mc A_{[c_1,c_2]}$ for $0\le c_1\le c_2<\infty$ is nonempty appears to depend on the universal prefix machine on which Kolmogorov complexity is based.

	\begin{que}\label{machine}
	Does there exist $0\le c_1\le c_2<\infty$ and a choice of universal machine underlying Kolmogorov complexity such that $\mc A_{[c_1,c_2]}$ has no effective enumeration?
	\end{que}

	\section{Families of $\Pi^{0}_1$ classes}\label{Pi01}

	\begin{df}[\cite{BrC08}]
	Let $\mc C$ be a family of closed subsets of $2^\omega$. We say that $\mc C$ has a \emph{computable enumeration} if there is a uniformly computable collection $\{T_e\}_{e\in\omega}$ of trees $T_e\subseteq 2^{<\omega}$ (that is, $\{\la \sigma,e\ra : \sigma\in T_e\}$ is computable, and $\sigma^\frown\tau\in T_e$ implies $\sigma\in T_e$) such that $\mc C=\{[T_e]: e\in\omega\}$.  
	\end{df}

	\begin{df}[\cite{BrC08}]
	Let $\mc C$ be a family of closed subsets of $2^\omega$. We say that $\mc C$ has an \emph{effective enumeration} if there is a $\Pi^0_1$ set $S\subseteq 2^\omega\times\omega$, such that $\mc C=\{ \{X: (X,e)\in S\} : e\in\omega\}$. 
	\end{df}

	\begin{pro}\label{6pm}
	Let $\mc C$ be a family of closed subsets of $2^\omega$. The following are equivalent:
	\begin{enumerate}
	\item[(1)] $\mc C$ has a {computable enumeration};
	\item[(2)] $\mc C$ has an {effective enumeration}.
	\end{enumerate}
	\end{pro}
	\begin{proof}
	(1) implies (2): Let $\{T_e\}_{e\in\omega}$ be given, and define $$S=\{(X,e): \forall n\,\,X\restrict n \in T_e\}.$$

	(2) implies (1): Let $S$ be given, let $\Phi_a$ be a Turing functional such that 
	$(X,e)\in S\Iff \Phi_a^X(e)\uparrow$, and let $T_e=\{\sigma\in 2^{<\omega}: \Phi_{a,|\sigma|}^\sigma(e)\uparrow\}$.
	\end{proof}

	\noindent In light of Proposition \ref{6pm}, we may use either notion. Note that if $C$ belongs to a family as in Proposition 1 then $C$ is a $\Pi^0_1$ class.

	\subsection{Existence of numberings}

	\begin{thm}\label{1}  Let $P\subseteq 2^\omega$, let $\mc C_P$ be the collection of all $\Pi^0_1$ classes contained in $P$, and let $\mathcal N_P$ be the collection of all nonempty $\Pi^0_1$ classes contained in $P$. Assume $P$ has the following properties: 
	\begin{enumerate}
	\item[(i)] $P$ is co-dense: no cone $[\sigma]$, $\sigma\in 2^{<\omega}$, is contained in $P$;
	\item[(ii)] $P$ is closed under shifts: if $x\in P$ then $\sigma^{\frown}x\in P$;
	\item[(iii)] $\mathcal N_P\ne\emptyset$. 
	\end{enumerate}
	Then there is no effective numbering of either $\mc C_P$ or $\mathcal N_P$.
	\end{thm}
	\begin{proof}
	If there is a numbering of $\mathcal N_P$ then there is one of $\mc C_P$, because if $\emptyset\in\mathcal C_P$ (as is always the case) we may simply add an index of $\emptyset$ to the numbering. Thus it suffices to show that there is no effective numbering of $\mathcal C_P$.
	Suppose to the contrary  that $e\mapsto [T_{e}]$ enumerates the family of $\Pi^{0}_{1}$ classes in $\mc C_P$. By (iii), we may assume $[T_0]\ne\emptyset$. By (i), $T_0$ has infinitely many dead ends. Let the dead ends of $T_0$ be listed in a computable way (for instance, by length-lexicographic order), as $\sigma_n$, $n\in\omega$.  By (i) again, we may let $\tau_n$ be the least extension of $\sigma_n$ which extends a dead end of $T_n$.  Define a computable tree $T$ by putting $T_0$ above $\tau_n$. That is, let $[T]\inter [\tau_n]=[\tau_n T_0]$ and $[T]=[T_0]\cup \bigcup_n [\tau_n T_0]$.  By (ii), the resulting class $[T]$ belongs to $\mc C_P$. Since $[T_0]\ne\nil$, $[T]\inter [\tau_n]\ne\emptyset=[T_n]\inter [\tau_n]$, so $[T]$ is not contained in or equal to any $[T_n]$.
	\end{proof}

	All assumptions (i), (ii), (iii) of Theorem \ref{1} are necessary: consider $P=2^\omega$, $P=\{x\}$, where $x$ is a single computable real, and $P=\emptyset$, respectively.

	\begin{cor}\label{Simpson.question}
	The following families of $\Pi^0_1$ classes have no effective numbering:
	\begin{enumerate}
	\item $\Pi^0_1$ classes containing only Martin-L\"of random reals;
	\item special $\Pi^0_1$ classes (those containing only non-computable reals);
	\item $\Pi^0_1$ classes containing only reals $x$ such that the Muchnik degree \cite{Simpnik} of $\{x\}$ is above a fixed nonzero Muchnik degree;
	\item $\Pi^0_1$ classes containing only finite (or only co-finite) subsets of $\omega$.
	\end{enumerate}
	\end{cor}

	\begin{pro}\label{lastmin}
	\begin{enumerate}
	\item[(1)] The family of all $\Pi^0_1$ classes containing only reals that are Martin-L\"of random with respect to a fixed randomness constant {is} effectively enumerable. 
	\item[(2)] The family of all $\Sigma^0_2$ classes containing only Martin-L\"of random reals is effectively enumerable. 
	\end{enumerate}
	\end{pro}
	\begin{proof}
	(1). We enumerate all $\Pi^0_1$ classes as $\{P_i\}_{i\in\omega}$ and let $$Q_i=P_i\cap \{x: \forall n\,\, K(x \restrict n)\ge n-c\}.$$ Then $\{Q_i\}_{i\in\omega}$ is an enumeration of all $\Pi^0_1$ classes containing only reals that are Martin-L\"of random with randomness constant $c$.  Part (2) is analogous. 
	\end{proof}

	We may sum up the situation by stating that it is only the mixture of $\Pi^0_1$ and $\Sigma^0_2$ classes that leads to the negative result of Corollary \ref{Simpson.question}(1). The proof of Theorem \ref{1} for the case in Corollary \ref{Simpson.question}(1) proves the following basic property of Martin-L\"of tests. 

	\begin{cor}
	For each Martin-L\"of test $\{U_n\}_{n\in\omega}$ there is a $\Sigma^0_1$ class $V$ containing all non-Martin-L\"of random reals but containing no set $U_n$, $n\in\omega$. 
	\end{cor}

	As is well-known, all $\Pi^0_1$ classes containing Martin-L\"of random reals have positive measure. In contrast to Corollary \ref{Simpson.question}(1), such classes can be effectively enumerated:

	\begin{thm}\label{positive.measure.Pi01}
	There is an effective numbering of the $\Pi^0_1$ classes of positive measure. 
	\end{thm}

	\begin{proof}
	It suffices to enumerate, uniformly in $n\in\omega$, all $\Pi^{0}_{1}$ classes of measure at least $r:=\frac{1}{n}$.  To accomplish this, let $e\mapsto W_{e}$ be an effective numbering of  $\Sigma^{0}_{1}$ sets of strings, which gives rise to all $\Pi^{0}_{1}$ classes.  That is, if $P$ is a $\Pi^{0}_{1}$ class, then $P=2^{\omega}\setminus [W_e]^\preceq$ for some $e$.  Modify this enumeration so that strings enumerate into each $W_{e}$ so long as the overall measure never surpasses $1-r$.  More precisely, if, at some stage $s>0$, some $\sigma$ is supposed to enter $W_{e,s}$ but this causes the measure of $[W_e]^\preceq$ to surpass $1-r$, then we hereafter discontinue to enumerate strings into $W_{e}$; call this modified set $\widehat{W}_{e;n}$.  It follows that $e\mapsto\widehat{W}_{e;n}$ is a numbering that gives rise to all $\Sigma^{0}_{1}$ classes of measure at most $1-\frac{1}{n}$. Then 
	the sequence of sets $\left\{\left[\widehat W_{e;n}\right]^\preceq\right\}$ for ${\la e,n\ra\in \omega\times\omega}$ is an effective enumeration of the $\Sigma^0_1$ classes of measure less than 1.
	\end{proof}

	This contrasts with the result of \cite{BrC08} that there is no effective numbering of the $\Pi^0_1$ classes of measure zero.
	We next show that any effectively enumerable family of $\Pi^0_1$ classes containing all the clopen classes has a Friedberg numbering. In fact, we show something slightly stronger. 

	\begin{df}
	The \emph{optimal covering} of $S\subseteq 2^{<\omega}$ is
	$$O=O_S=\{\sigma : [\sigma]\subseteq [S]^\preceq \And \neg (\exists \tau\prec\sigma) ([\tau]\subseteq [S]^\preceq)\}.$$ 
	Let $\mathfrak A$ be the family of all sets $O$ that have odd cardinality and are optimal coverings of sets $S$. 
	\end{df}

	\begin{thm}\label{all}
	Any effectively enumerable family of $\Sigma^0_1$ classes $\mc F$ with $\mc F\supseteq\{[O]^\preceq: O\in\mf A\}$ has a Friedberg numbering.
	\end{thm}
	\begin{proof}
	For a set $Z\subseteq 2^{<\omega}$, we say that $Z$ is \emph{filter closed} if $Z$ is closed under extensions ($\sigma\in Z\Implies \sigma^\frown\tau\in Z$) and such that whenever both $\sigma^\frown 0$ and $\sigma^\frown 1$ are in $Z$ then $\sigma\in Z$. The \emph{filter closure} of $Y$ is the intersection of all filter closed sets containing $Y$ and is denoted by $Y^\uparrow$. 

	Since $\mc F$ is effectively enumerable, we may let $e\mapsto Y_e$ be a numbering of all filter closed sets of strings with $[Y_e]^\preceq\in \mc F$.  Since $Y_e\ne Y_{e'}$ implies $[Y_e]^{\preceq}\ne [Y_{e'}]^{\preceq}$, it suffices to injectively enumerate these sets $Y_e$. Let 
 
	$$L_1= \left\{O^\uparrow: O\in \mf A\right\}\text{ and }L_2=\{Y_e: Y_e\not\in L_1\}.$$

	It is clear that $L_1$ is injectively enumerable. By the assumption of the theorem, each $[O^\uparrow]^\preceq\in\mc F$. It is also clear that each finite subset of any $Y\in L_2$ is contained in {infinitely many} $O^\uparrow\in L_1$. 

	We claim that $L_2$ has an effectively enumeration $\{Y^*_e\}_{e\in\omega}$, to be constructed below. Fix $e$ and let $Y_e=\{\sigma_n\}_{n\in\omega}$ in order of enumeration.

	$S\subseteq 2^{<\omega}$ is an \emph{acceptable family} if its optimal covering $O$ has finite even cardinality. In particular $O\not\in\mf A$. We say that stage $n$ is \emph{good} if $\sigma_n$ has greater length than any member of $O_n=O_{S_n}$ for $S_n=\{\sigma_0,\ldots,\sigma_{n-1}\}$ and does not extend any member of $O_n$.

	\noindent\emph{Construction.}
	We will construct $Y^*_e$ as $Y^*_e=\bigcup_{n\in\omega} Y_{e,n}$ for uniformly computable sets $Y_{e,n}$. We set $Y_{e,-1}=\nil$. Suppose $n\ge 0$.
	\noindent If stage $n$ is not good, we keep $Y_{e,n}=Y_{e,n-1}$.

	\noindent If stage $n$ is good, there are two cases.

	\noindent Case a. $S_n$ is an acceptable family. Then let $Y_{e,n}$ be the filter closure of $O_n$. 

	\noindent Case b. Otherwise. Then let $Y_{e,n}$ be the filter closure of $O_n\cup \{\sigma_n\}$. 

	\noindent We separately enumerate all sets generated from any acceptable family whose optimal covering has finite even cardinality.  (*)

	\noindent \emph{End of Construction.} 

	\noindent \emph{Verification.} 
	Note that in both Case a and Case b, $Y_{e,n}$ is the filter closure of an acceptable family, so we do not enumerate any member of $L_1$. By (*), it therefore suffices to show that we enumerate all sets generated from an infinite family, i.e. non-clopen sets, and that each $Y^*_e$ is some $Y_{e'}$. 

	If $Y_e$ is not clopen then there are infinitely many good stages. Then in the end $Y^*_e=Y_e$, because $\sigma_n$ is covered either right away (case b) or at the next good stage (case b).
	\end{proof}

	\begin{cor}
	The family of all $\Sigma^0_1$ classes of measure less than one, or equivalently $\Pi^0_1$ classes of positive measure, has a Friedberg numbering.
	\end{cor}


\begin{bibdiv}
		\begin{biblist}
			\bib{BrC08}{article}{
			   author={Brodhead, Paul},
			   author={Cenzer, Douglas},
			   title={Effectively closed sets and enumerations},
			   journal={Arch.\ Math.\ Logic},
			   volume={46},
			   date={2008},
			   number={7-8},
			   pages={565--582},
			   issn={0933-5846},
			   review={\MR{2395559 (2009b:03119)}},
			}

			\bib{Ers99}{article}{
			   author={Ershov, Yuri L.},
			   title={Theory of numberings},
			   conference={
			      title={Handbook of computability theory},
			   },
			   book={
			      series={Stud.\ Logic Found.\ Math.},
			      volume={140},
			      publisher={North-Holland},
			      place={Amsterdam},
			   },
			   date={1999},
			   pages={473--503},
			   review={\MR{1720731 (2000j:03060)}},
			}

			\bib{Fri58}{article}{
			   author={Friedberg, Richard M.},
			   title={Three theorems on recursive enumeration. I.\ Decomposition. II.\
			   Maximal set. III.\ Enumeration without duplication},
			   journal={J.\ Symb.\ Logic},
			   volume={23},
			   date={1958},
			   pages={309--316},
			   issn={0022-4812},
			   review={\MR{0109125 (22 \#13)}},
			}

			\bib{LV}{book}{
			   author={Li, Ming},
			   author={Vit{\'a}nyi, Paul},
			   title={An introduction to Kolmogorov complexity and its applications},
			   series={Graduate Texts in Computer Science},
			   edition={2},
			   publisher={Springer-Verlag},
			   place={New York},
			   date={1997},
			   pages={xx+637},
			   isbn={0-387-94868-6},
			   review={\MR{1438307 (97k:68086)}},
			}

			\bib{Kummer}{article}{
			   author={Kummer, Martin},
			   title={An easy priority-free proof of a theorem of Friedberg},
			   journal={Theoret.\ Comput.\ Sci.},
			   volume={74},
			   date={1990},
			   number={2},
			   pages={249--251},
			   issn={0304-3975},
			   review={\MR{1067521 (91h:03056)}},
			}

			\bib{Nies:book}{book}{
				author={Nies, Andr\'e},
				title={Computability and randomness},
				publisher={Oxford University Press},
				year={2009},
			}

			\bib{Simpnik}{article}{
			   author={Simpson, Stephen G.},
			   title={An extension of the recursively enumerable Turing degrees},
			   journal={J. Lond. Math. Soc. (2)},
			   volume={75},
			   date={2007},
			   number={2},
			   pages={287--297},
			   issn={0024-6107},
			   review={\MR{2340228 (2008d:03041)}},
			}
		\end{biblist}
	\end{bibdiv}
\end{document}